\documentclass[reqno, 12pt]{amsart}
\usepackage{amsmath}
\usepackage{amsfonts}
\usepackage{amssymb}
\usepackage{amsthm}
\usepackage{mathrsfs}
\usepackage{bbm}
\usepackage{graphicx, color}

\usepackage{adjustbox}

\usepackage{bmpsize}

\usepackage{mathtools}

\usepackage{enumerate}

\definecolor{refkey}{rgb}{0.9451,0.2706,0.4941}
\definecolor{labelkey}{rgb}{0.9451,0.2706,0.4941}

\usepackage{color}

\usepackage{subfiles}
\usepackage[colorlinks, linkcolor=blue, citecolor=red, urlcolor=cyan,pagebackref]{hyperref}

\allowdisplaybreaks

\numberwithin{equation}{section}

\usepackage[colorlinks, linkcolor=blue, citecolor=red, urlcolor=cyan,pagebackref]{hyperref}%逆引き有効化

\usepackage{cleveref}
\crefname{thm}{Theorem}{Theorems}
\crefname{cor}{Corollary}{Corollaries}
\crefname{lem}{Lemma}{Lemmas}
\crefname{sublem}{Sublemma}{Sublemmas}
\crefname{prop}{Proposition}{Propositions}
\crefname{dfn}{Definition}{Definitions}
\crefname{defi}{Definition}{Definitions}
\crefname{ex}{Example}{Examples}
\crefname{claim}{Claim}{Claims}
\crefname{conj}{Conjecture}{Conjectures}
\crefname{conv}{Notation}{Notations}
\crefname{rem}{Remark}{Remarks}
\crefname{rmk}{Remark}{Remarks}
\crefname{prob}{Problem}{Problems}
\crefname{figure}{Figure}{Figures}
\crefname{table}{Table}{Tables}
\crefname{section}{Section}{Sections}
\crefname{subsection}{Section}{Sections}
\crefname{appendix}{Appendix}{Appendices}
\crefname{introthm}{Theorem}{Theorems}
\crefname{introcor}{Corollary}{Corollaries}
\crefname{introconj}{Conjecture}{Conjectures}

\newtheorem{thm}{Theorem}[section]
\newtheorem{prop}[thm]{Proposition}
\newtheorem{cor}[thm]{Corollary}

\newtheorem{introthm}{Theorem}

\newtheorem{introconj}[introthm]{Conjecture}

\theoremstyle{definition}
\newtheorem{dfn}[thm]{Definition}

\newtheorem{ex}[thm]{Example}

\newtheorem{conj}[thm]{Conjecture}

\theoremstyle{remark}
\newtheorem{rmk}[thm]{Remark}

\newcommand{\bZ}{\mathbb{Z}}

\newcommand{\bR}{\mathbb{R}}
\newcommand{\bC}{\mathbb{C}}

\newcommand{\bD}{\mathbb{D}}

\newcommand{\cM}{\mathcal{M}}

\newcommand{\cR}{\mathcal{R}}

\newcommand{\cT}{\mathcal{T}}

\newcommand{\al}{\alpha}
\newcommand{\sS}{\mathscr{S}}

\newcommand{\sSq}{\mathscr{S}_q}

\newcommand{\sfB}{\mathsf{B}}

\newcommand{\barT}{\overline{T}}

\title[On positivity of Roger--Yang skein algebras]{On positivity of Roger--Yang skein algebras}

\author[Hiroaki Karuo]{Hiroaki Karuo}
\address{Hiroaki Karuo, Department of Mathematics, Gakushuin University, Mejiro, Toshima-ku, Tokyo, Japan.}
\email{hiroaki.karuo@gakushuin.ac.jp}

\begin{document}
\maketitle

\begin{abstract}
We generalize the positivity conjecture on (Kauffman bracket) skein algebras to Roger--Yang skein algebras. 
To generalize it, we use Chebyshev polynomials of the first kind to give candidates of positive bases.
Moreover, the polynomials form a lower bound in the sense of \cite{Le18} and \cite{LTY}. 
We also discuss a relation between the polynomials and the centers of Roger--Yang skein algebras when the quantum parameter is a complex root of unity.
\end{abstract}

\tableofcontents

\section{Introduction}
\paragraph{{\bf Roger--Yang skein algebras and related works}} For $\cR$ a commutative unital ring with $q^{\pm1}$ (or $q^{\pm1/2}$), the (Kauffman bracket) skein algebra of a punctured surface $\Sigma$ is a free $\cR$-module generated by isotopy classes of framed links in $\Sigma\times (0,1)$ subject to some skein relations with multiplication defined by stacking with respect to the second factor $(0,1)$. 
For a punctured surface $\Sigma$, Roger--Yang \cite{RY} gave a generalization of skein algebras, called Roger--Yang skein algebras and denoted by $\sS_q^{{\rm RY}}(\Sigma)$, to give a relation between $\sS_1^{{\rm RY}}(\Sigma)$ and decorated Teichm\"uller spaces, where $\sS_1^{{\rm RY}}(\Sigma)$ is $\sS_q^{{\rm RY}}(\Sigma)$ with $\cR=\bC$ and $q^{1/2}=1$. 
Moreover, Moon--Wong \cite{MW21}, \cite{MW} formulated some relations between a subalgebra of $\sS_1^{{\rm RY}}(\Sigma)$ and the (upper) cluster algebra of $\Sigma$. 
In this paper, we use a slightly different version of the original Roger--Yang skein algebra $\sS_q^{{\rm RY}}(\Sigma)$, denoted by $\sSq(\Sigma)$. 
It is an $\cR$-module generated by isotopy classes of framed links and arcs in $\Sigma\times (0,1)$ subject to some relations with multiplication defined by stacking.  Note that the original Roger--Yang skein algebra $\sS_q^{{\rm RY}}(\Sigma)$ is contained in the tensor product $\sSq(\Sigma)\otimes_\cR \cR[p^{\pm 1/2}\mid \text{$p$ is a puncture of $\Sigma$}]$, where the second factor is a Laurent polynomial ring over $\cR$, see Proposition \ref{Prop_BKL}.

\paragraph{{\bf Positivity conjecture of skein algebras}}
In the case of $\cR=\bZ[q^{\pm1}]$, D. Thurston \cite{Thu} formulated the positivity conjecture of skein algebras. Here, the (original) positivity conjecture means certain bases of skein algebras have structure constants in the non-negative part $\bZ_{\geq0}[q^{\pm1}]$, where such bases are called \emph{positive bases}. 
The aim of this paper is to formulate a positivity conjecture of Roger--Yang skein algebras $\sSq(\Sigma)$. If the positivity conjecture holds, the positivity leads us to an expectation that there exists a (monoidal) categorification of Roger--Yang skein algebras similar to that of (Kauffman bracket) skein algebras \cite{QW21}.

Skein theoretic proof of the positivity conjecture of (Muller \cite{Mul}) skein algebras is known for 
\begin{itemize}
    \item the sphere with at most 3 punctures, 
    \item the torus \cite{FG00},
    \item the four-punctured sphere and the once-punctured torus \cite{Bou}, 
    \item marked disks and the marked annulus such that each boundary component has exactly one marked point \cite{IK23}, see also \cite{Yam}. 
\end{itemize}
Mandel--Qin affirmatively solved the positivity conjecture for a wide class of marked surfaces as a corollary by showing quantum theta functions correspond to the quantum bracelets basis \cite{MQ}. Note that their result includes the latter 2 cases in the above list. 
They also mentioned quantum tagged bracelets bases for punctured surfaces in the same paper. In particular, they defined such bases using principal coefficients and quantum cluster monomials. 
While Mandel--Qin's bases contain monomials realized by arcs incident to the same punctures, called cluster monomials in cluster algebras, we resolve such arcs using a defining relation in Roger--Yang skein algebras. 
This implies that the conjecture in this paper is different from the positivity conjecture in the context of cluster algebras.  
Inspired by the work, another aim of this paper is to give a geometric approach to quantum tagged bracelets bases through skein algebras.  
While the bases given in this paper may be related to Mandel--Qin's bases in some sense, we do not know explicit relations between them due to some differences mentioned above.

\paragraph{{\bf Main result and positivity conjecture}}
We show how to give a $\bZ[q^{\pm 1/2}]$-basis of $\sSq(\Sigma)$ from a pair of normalized sequences of polynomials (in Section \ref{Sec_posi}), where a sequence $(Q_n(x))_{n=0}^\infty$ of polynomials is \emph{normalized} if $Q_n(x)$ is a monic polynomial of degree $n$ for any $n\geq 0$. 
A pair of normalized sequences is \emph{positive} for a punctured surface $\Sigma$ if the $\bZ[q^{\pm 1/2}]$-basis obtained from the pair is a positive basis of $\sSq(\Sigma)$.

For the normalized sequences $(\barT_n(x))$ given in Section \ref{Sec_poly}, 
let $\sfB$ be the $\bZ[q^{\pm1/2}]$-basis of $\sSq(\Sigma)$ obtained from the sequences. 
Then, the following main result implies that $(\barT_n(x))$ forms a lower bound, in the sense of \cite{Le18} and \cite{LTY}, of any positive pairs of normalized sequences under some condition.  
\begin{introthm}[Theorem \ref{Thm_lower}]\label{Intro_thm}
If a pair of normalized sequences $((Q_n(x)), (R_n(x)))$ is positive for a punctured surface of genus $g\geq 1$ with at least 3 punctures, then each of $Q_n(x)$ and $R_n(x)$ is a $\bZ_{\geq0}[q^{\pm1/2}]$-linear combination of $\barT_i(x)$\ ($i=0,1,\dots, n$).
\end{introthm}
\noindent A remarkable observation is that we use ($\barT_n(x)$) instead of ($x^n$) for arcs, see Remark \ref{Rmk_non-triv}.

The following is a positivity conjecture of Roger--Yang skein algebras.
\begin{introconj}[Conjecture \ref{Conj_posiRY}]
The $\bZ[q^{\pm1/2}]$-basis $\sfB$ is positive for any punctured surfaces with at least 2 puncntures except for spheres with at most 3 punctures.
\end{introconj}
\noindent Note that one can consider a similar claim for a once-punctured surface with genus $g\geq1$ and the claim holds from the positivity of skein algebras. See Remark \ref{Rmk_conj} (1).

\paragraph{\textbf{Acknowledgments.}}
The author would like to thank Tsukasa Ishibashi, Thang L\^e, and Helen Wong for valuable comments. 
The author was partially supported by JSPS KAKENHI Grant Numbers JP22K20342, JP23K12976. 

\section{Chebyshev polynomials}\label{Sec_poly}
Chebyshev polynomials of the first kind $T_n(x)$ (resp. of the second kind $S_n(x)$) are defined from the following initial data and the recurrence relation; 
\begin{align*}
&T_0(x)=2,\ T_1(x)=x,\ T_{n+1}(x)=xT_{n}(x)-T_{n-1}(x)\ (n\geq 1),\\
&S_{-1}(x)=0,\ S_0(x)=1,\ S_1(x)=x,\ S_{n+1}(x)=xS_{n}(x)-S_{n-1}(x)\ (n\geq 1).
\end{align*}

It is well-known that 
\begin{align}
&T_m(x)\cdot T_n(x)=T_{m+n}(x)+T_{|m-n|}(x).\label{rel:Chebyshev}
\end{align}
For $n,m\geq 0$, one can easily show that 
\begin{align}
&T_{2n}(x)=T_{n}(x^2-2),\label{rel:Cheby_even}\\
&(S_{n}(x)-S_{n-1}(x))(S_{m}(x)-S_{m-1}(x))(x+2)=T_{n+m+1}(x)+T_{|n-m|}(x),\label{rel:Cheby_SS}\\
&T_n(x)(S_{m}(x)-S_{m-1}(x))=(S_{n+m}(x)-S_{n+m-1}(x))+(S_{\varepsilon(n,m)}(x)-S_{\varepsilon(n,m)-1}(x))\label{rel:Cheby_TS}
\end{align}
with 
$$\varepsilon(n,m)
=\begin{cases}
n-m-1&\text{if $n>m$},\\
m-n&\text{if $n\leq m$}
\end{cases}.$$

For later convenience, set 
\begin{align}
\text{$\barT_n(x):=T_n(x)\ (n\geq 1)$ and $\barT_0(x):=1$.} \label{def;barT}
\end{align}
For an integer $n\geq 1$, one can easily show that 
\begin{align}
\barT_{2n+1}(x)=\big(S_{n}-S_{n-1}\big)(x^2-2)\cdot x.\label{rel;important}
\end{align}

%Consider the polynomials $P_n(x)$ in $\bZ[q^{\pm 1/2}][x]$,
%\begin{align}
%&P_n(x)
%=\begin{cases}
%\barT_{n/2}(x^2-2)&\text{if $n$ is even}, \\
%\big(S_{(n-1)/2}-S_{(n-3)/2}\big)(x^2-2)\cdot x&\text{if $n$ is odd.}\label{Def_poly}
%\end{cases}
%\end{align}
%Note that the coefficients with respect to $x$ are in $\bZ$. 

\begin{ex}
$S_2(x)-S_1(x)=x^2-x-1,\quad S_3(x)-S_2(x)=x^3-x^2-2x+1$. 
\end{ex}

A sequence $(Q_n(x))$ of polynomials over $\bZ[q^{\pm 1/2}]$ is \emph{normalized} if for each $n\in \bZ_{\geq0}$, $Q_n(x)$ is a monic polynomial of degree $n$, where monic means the coefficient of the highest term is $1$. %Note that the sequence $(P_n(x))$ defined above is normalized. 

\begin{dfn}
For two normalized sequences $(Q_i(x)), (R_i(x))$, we write $(Q_i(x))\geq (R_i(x))$ if $Q_n(x)$ can be presented as a $\bZ_{\geq0}[q^{\pm1/2}]$-linear combination of $R_i(x)\ (i=0,1,\dots, n)$ for any $n$. 
\end{dfn}

\begin{rmk}\label{Rmk_ineq}
It is easy to see that, for any positive integer $n$, we have 
\begin{align*}
&\barT_0(x)=S_0(x)-S_{-1}(x)=S_{0}(x)=1,\\ %(S_1(x)-S_{0}(x))+(S_0(x)-S_{-1}(x))=\barT_0(1)=S_{1}(x)=x\\
&\barT_n(x)=(S_n(x)-S_{n-1}(x))+(S_{n-1}(x)-S_{n-2}(x)),\\
&S_{2n}(x)=\sum_{i=0}^{n}\barT_{2n-2i}(x),\quad S_{2n+1}(x)=\sum_{i=0}^{n}\barT_{2n-2i+1}(x). 
\end{align*}
These imply that 
$(S_n(x)-S_{n-1}(x))\leq (\barT_n(x))\leq (S_n(x))\leq (x^n)$. 
\end{rmk} 

\section{Roger--Yang skein algebras and their classical versions}\label{Sec_RY}
For a surface $\Sigma$, let $\cM$ be a finite subset of the interior of $\Sigma$. 
A \emph{marked surface} is a pair $(\Sigma,\cM)$. 
A \emph{punctured surface} is a marked surface with a closed surface $\Sigma$ and a non-empty set $\cM$. 
We call each point of $\cM$ a \emph{puncture}. 
In the following, suppose any punctured surface has at least 2 punctures.

A punctured surface is \emph{small} if the complement of a small neighborhood of any arc connecting two different punctures is either the disk or once-punctured disk. 
A punctured surface is \emph{planar} if the underlying surface is the sphere. 
Note that any small punctured surface is planar.

\paragraph{\textbf{$\cM$-tangles and diagrams.}}
An $\cM$-tangle $\al$ of $(\Sigma, \cM)$ is an embedded framed loops and arcs such that 
\begin{enumerate}
    \item $\partial\al\cap (\Sigma\times (0,1))\subset \cM\times(0,1)$, and
    \item the normal vectors at the endpoints are vertical pointing to $1$. 
\end{enumerate}

We regard the empty set as an $\cM$-tangle. 
Two $\cM$-tangles are \emph{isotopic} if they are isotopic in the class of $\cM$-tangles. 

By isotopy, we deform an $\cM$-tangle to one with vertical framing. 
Then, we project it onto $\Sigma\times \{1/2\}$ (with slight perturbation if need) so that 
\begin{enumerate}
    \item the image has only transverse crossings as singularities,
    \item in the interior of $\Sigma$, the image has only transversal double crossings, and
    \item each crossing is equipped with over/under information with respect to the second factor. 
\end{enumerate}

A loop on $\Sigma$ is \emph{trivial} if it bounds a disk or a once-punctured disk. 
An $\cM$-tangle diagram $\al$ is \emph{simple} if it has no crossings, no trivial loops, and the boundary of a small neighborhood of each puncture intersects with $\al$ at most once. If a simple $\cM$-tangle diagram consists of a loop (resp. an arc), we call it a \emph{simple loop} (resp. a \emph{simple arc}).

%Note that we do not consider any spirals as $\cM$-tangle diagrams. 

\paragraph{\textbf{Roger--Yang skein algebras.}}
Let $\cR$ be a commutative unital ring with a distinguished invertible $q^{1/2}$. 
Let $(\Sigma,\cM)$ be an oriented punctured surface.

The \emph{Roger--Yang skein algebra} of $(\Sigma,\cM)$ is the $\cR$-module generated by isotopy classes of $\cM$-tangles in $\Sigma\times (0,1)$ subject to the relations (A)--(D) with multiplication defined by stacking with respect to $(0,1)$, which is denoted by $\sSq(\Sigma)$. 

\begin{align*}
&({\rm A}) \begin{array}{c}\includegraphics[scale=0.18]{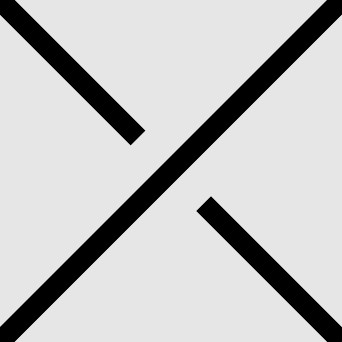}\end{array}=q\begin{array}{c}\includegraphics[scale=0.18]{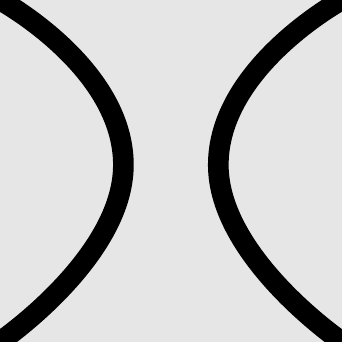}\end{array}+q^{-1}\begin{array}{c}\includegraphics[scale=0.18]{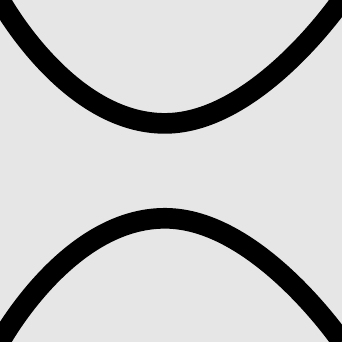}\end{array},\qquad ({\rm B}) \begin{array}{c}\includegraphics[scale=0.18]{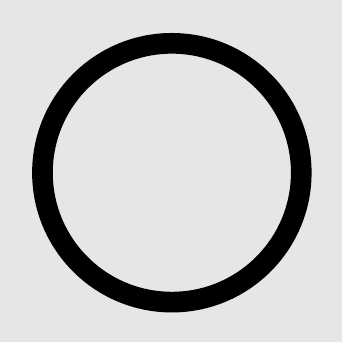}\end{array}=(-q^2-q^{-2})\begin{array}{c}\includegraphics[scale=0.18]{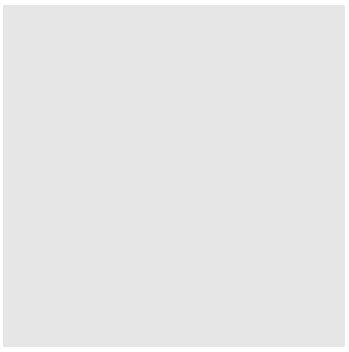}\end{array},\\
&({\rm C}) \begin{array}{c}\includegraphics[scale=0.18]{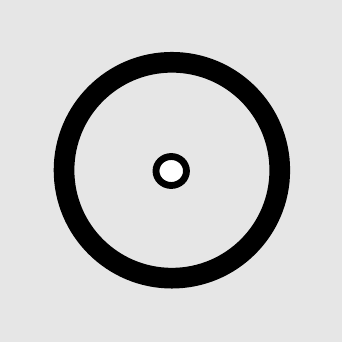}\end{array}=(q+q^{-1})\begin{array}{c}\includegraphics[scale=0.18]{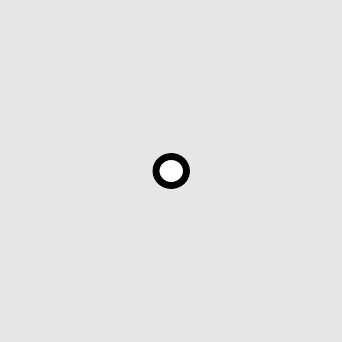}\end{array},\qquad
({\rm D}) \begin{array}{c}\includegraphics[scale=0.18]{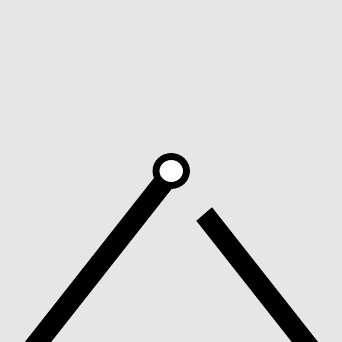}\end{array}=q^{1/2}\begin{array}{c}\includegraphics[scale=0.18]{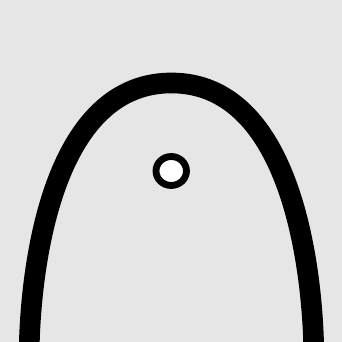}\end{array}+q^{-1/2}\begin{array}{c}\includegraphics[scale=0.18]{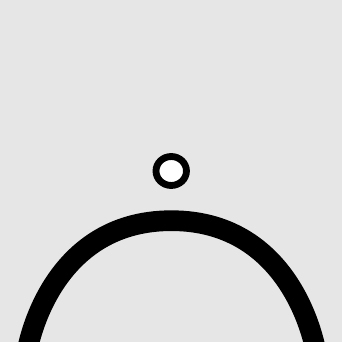}\end{array},
\end{align*}
where we depict a puncture as a white vertex.

It is known that the following set is a basis of $\sSq(\Sigma)$ as an $\cR$-module \cite{RY}, \cite{BKL}. 
\begin{align}
\bar{\mathsf{B}}=\{\text{isotopy classes of simple $\cM$-tangle diagrams on $(\Sigma,\cM)$} \}\label{basis}
\end{align}

There are some relations between $\sSq(\Sigma)$ and the original Roger--Yang skein algebra $\sS_q^{{\rm RY}}(\Sigma)$, where $\sS_q^{{\rm RY}}(\Sigma)$ is defined by replacing $\cR$ with the Laurent polynomial ring $\cR[p^{\pm1}\mid \text{$p\in \cM$}]$ and using (D') instead of (D) in the definition of $\sSq(\Sigma)$; 
\begin{align*}
({\rm D'}) \begin{array}{c}\includegraphics[scale=0.18]{draws/crossing_int.pdf}\end{array}=p^{-1}\Big(q^{1/2}\begin{array}{c}\includegraphics[scale=0.18]{draws/intarcs_over.pdf}\end{array}+q^{-1/2}\begin{array}{c}\includegraphics[scale=0.18]{draws/intarcs_under.pdf}\end{array}\Big)\quad
\text{around the puncture $p$.}
\end{align*}
\begin{prop}[\cite{BKL}]\label{Prop_BKL}
There is an algebra embedding 
$$\iota\colon \sS_q^{{\rm RY}}(\Sigma)\hookrightarrow \sSq(\Sigma)\otimes_\cR \cR[p^{\pm1/2}\mid \text{$p\in \cM$}]$$ 
given by $\iota(\al)=(\prod_{p\in \cM}p^{-I(\al,\partial p)/2})\al$, where $I(\al,\partial p)$ is the geometric intersection number of $\al$ and the boundary of a small neighborhood of $p$ in $\Sigma$. 
\end{prop}
Although we consider a positivity conjecture of Roger--Yang skein algebras $\sSq(\Sigma)$ in the paper, we can also treat that of $\sS_q^{{\rm RY}}(\Sigma)$ with a slight modification.  See Remark \ref{Rmk_conj} (2).

Now, let $\sS_1^{{\rm RY}}(\Sigma)$ denote $\sS_q^{{\rm RY}}(\Sigma)$ with
$\cR=\bC$ and $q^{1/2}=1$, called the \emph{curve algebra} \cite{MW21}. Then, some important results on original Roger--Yang skein algebras are the following. 
\begin{thm}[\cite{RY}, \cite{MW21}, \cite{MW}, \cite{BKL}]
There is a Poisson algebra homomorphism $\Phi\colon \sS_1^{{\rm RY}}(\Sigma)\to C^{\infty}[\cT^{{\rm dec}}(\Sigma)]$, where $C^{\infty}[\cT^{{\rm dec}}(\Sigma)]$ denotes the ring of $\bC$-valued smooth functions on the decorated Teichm\"uller space of $\Sigma$. 
Moreover, $\Phi$ is injective. 
\end{thm}

\begin{thm}[\cite{MW}, see also \cite{FST}]
The algebra $\sS_1^{{\rm RY}}(\Sigma)$ has a subalgebra which is sandwiched by the cluster algebra and the upper cluster algebra of $\Sigma$. 
\end{thm}

\section{Bases of Roger--Yang skein algebras and positivity}\label{Sec_posi}
In the following, suppose that $\cR=\bZ[q^{\pm 1/2}]$. 

Let $\al$ be a simple arc connecting different punctures.
Then, we have 
\begin{align}
p_\al=\al^2-2,\label{formula:squared}
\end{align}
where $p_\al$ is a simple loop on $\Sigma\setminus \cM$ bounding a small neighborhood of $\al$.

%From (\ref{Def_poly}) we have 
%\begin{align*}
%P_n(\al)=
%\begin{cases}
%\barT_{n/2}(p_\al)&\text{if $n$ is even}, \\
%\big(S_{(n-1)/2}(p_\al)-S_{(n-3)/2}(p_\al)\big)\cdot \al &\text{if $n$ is odd,}
%\end{cases}
%\end{align*}

For a non-small punctured surface $(\Sigma,\cM)$, we consider the set
\begin{align}
\sfB=\{\prod_{k}\barT_{m_k}(\gamma_k)\prod_{\ell}\barT_{n_\ell}(\al_\ell)\mid  \text{$m_k,n_\ell\in \bZ_{\geq0}$ and the following conditions}\};\label{def;basis}
\end{align}
\begin{enumerate}
    \item \text{$\gamma_k$ is a simple loop such that any component of $\Sigma\setminus \gamma_k$ is not a disk with 2 punctures},
    \item \text{$\al_\ell$ is a simple arc}, 
    \item \text{any two of $\gamma_k, \al_\ell$ are not isotopic}, and
    \item \text{$\prod_{k}\gamma_k \prod_{\ell}\al_\ell$ is a simple $\cM$-tangle diagram}. 
\end{enumerate}
Only for the four-punctured sphere, we add a condition to the above;
\begin{align}
\text {for two simple arcs $\al_1, \al_2$ satisfying the above condition, 
$n_1$ or $n_2$ is 1.}\label{condi;4-punc}
\end{align}

For $\mathbf{Q}=((Q_n(x)),(R_n(x)))$ a pair of normalized sequences, let $\mathsf{B}_{\mathbf{Q}}$ be the set obtained from (\ref{def;basis}) by replacing $\barT_{m_k}(\gamma_k)$ and $\barT_{n_\ell}(\al_\ell)$ with $Q_{m_k}(\gamma_k)$ and $R_{n_\ell}(\al_\ell)$ respectively. 
%where the first $(\barT_n(x))$ is for loops and the second $(\barT_n(x))$ is for arcs. 

\begin{prop}\label{Prop_basis}
For a non-small punctured surface $(\Sigma,\cM)$ and a pair $\mathbf{Q}$ of normalized sequences, the set $\sfB_{\mathbf{Q}}$ is a $\bZ[q^{\pm1/2}]$-basis of $\sSq(\Sigma)$, especially so is $\sfB$.  
\end{prop}
\begin{proof}
Now we consider a non-small punctured surface $(\Sigma,\cM)$ except for the four-punctured sphere. 
Consider the normalized sequence $(U_n(x))$ defined by 
$$U_n(x)=
\begin{cases}
(x^2-2)^{n/2}&\text{if $n$ is even,}\\
(x^2-2)^{(n-1)/2}\cdot x&\text{if $n$ is odd,} 
\end{cases}$$
and the pair $\mathbf{U}=((x^n), (U_n(x)))$ of normalized sequences. 
With (\ref{formula:squared}), it is easy to see that any element in $\sfB_{\mathbf{U}}$ is a simple $\cM$-tangle diagram, i.e. $\sfB_{\mathbf{U}}\subset \bar{\sfB}$, where $\bar{\sfB}$ was defined as (\ref{basis}).

Let $\beta$ be a simple $\cM$-tangle diagram. 
For a simple loop $\gamma$ which satisfies the condition (1) in (\ref{def;basis}), we obtain parallel copies of $\gamma$ as $\gamma^{n}$. 
For a simple loop $\gamma$ which does not satisfy the condition (1) in (\ref{def;basis}), then there is an arc $\al$ such that $p_\al=U_2(\al)$ is isotopic to $\gamma$. Then, parallel copies of $\gamma$ are obtained as $U_n(\al)$; the parity of $n$ depends on the absence of $\al$ in $\gamma$. 
These imply that any simple $\cM$-tangle diagram is an element of $\sfB_{\mathbf{U}}$. 
Hence, $\sfB_{\mathbf{U}}=\bar{\sfB}$, i.e. $\sfB_{\mathbf{U}}$ is a $\bZ[q^{\pm 1/2}]$-basis of $\sSq(\Sigma)$.

For $\mathbf{Q}=((Q_n(x)), (R_n(x)))$ a pair of normalized sequences, 
since both of $(Q_n(x))$ and $(R_n(x))$ are $\bZ[q^{\pm1/2}]$-bases of $\bZ[q^{\pm1/2}][x]$,
$x^n$ and $U_n(x)$ can be presented as linear sums of $Q_i(x)$ and $R_i(x)$ ($0\leq i\leq n$) respectively. This implies that any simple $\cM$-tangle diagram can be presented as a linear sum of $\sfB_{\mathbf{Q}}$. In particular, each element of $\sfB_{\mathbf{Q}}$ is a linear sum of simple $\cM$-tangle diagrams. Since $\bar{\sfB}$ is a basis, the elements in $\sfB_{\mathbf{Q}}$ are linearly independent. 
Hence, $\sfB_{\mathbf{Q}}$ is also a $\bZ[q^{\pm 1/2}]$-basis of $\sSq(\Sigma)$, especially so is $\sfB$.

%For $\mathbf{Q}=((Q_n(x)), (R_n(x)))$ a pair of normalized sequences, since both of $(Q_n(x))$ and $(R_n(x))$ are $\bZ[q^{\pm1/2}]$-bases of $\bZ[q^{\pm1/2}][x]$, $\sfB_{\mathbf{Q}}$ is also a $\bZ[q^{\pm 1/2}]$-basis of $\sSq(\Sigma)$, especially so is $\sfB$. 

One can easily show the case of the four-punctured sphere with the additional condition (\ref{condi;4-punc}). 
\end{proof}

\begin{rmk}
For an even integer $m$, $\barT_m(\al)$ does not contain any arc and this is defined by a Chebyshev polynomial of the first kind by inserting $p_\al$. It is compatible with the basis given in \cite{Thu} except for peripheral loops. \end{rmk}
Then, the following is a positivity conjecture of Roger--Yang skein algebras. 
\begin{conj}\label{Conj_posiRY}
For any non-small punctured surface ($\Sigma,\cM$) with $|\cM|\geq2$, the basis $\sfB$ is a positive basis of the Roger--Yang skein algebra $\sSq(\Sigma)$.
\end{conj}
\begin{rmk}\label{Rmk_conj}
\begin{enumerate}
    \item For any non-planar once-punctured surface, its Roger--Yang skein algebra has the basis ${\bf B}$ without arcs. More precisely, the basis corresponds with the bracelets basis of its (Kauffman bracket) skein algebra in \cite{Thu} except peripheral loops. The positivity of skein algebras and Relation (C) ensure the positivity of ${\bf B}$. 
    \item One can formulate a positivity conjecture of original Roger--Yang skein algebras $\sS_q^{{\rm RY}}(\Sigma)$ by replacing $x^2-2$ with $p r x^2-2$ in (\ref{rel;important}) $(p,r\in \cM)$ and multiplying a power of $p r$ to have a normalized sequence. Moreover, the obtained basis in the same manner as (\ref{def;basis}) is $\sfB$ and the positivity conjectures in both of $\sS_q^{{\rm RY}}(\Sigma)$ and $\sSq(\Sigma)$ are equivalent since $\{p^{\pm1}\mid p\in \cM\}$ does not affect the positivity. 
\end{enumerate}
\end{rmk}

\paragraph{{\bf Lower bound of positive bases}}
For a non-small punctured surface with at least 3 punctures, take an arc $\al$ connecting $p,p'\in \cM$ ($p\neq p'$). 
Fix an arc $\beta$ connecting $p,p''\in \cM$ ($p''\neq p,p'$) such that $\beta$ intersects with $p_{\al}$ just once, Figure \ref{Fig_local}. 
\begin{prop}[{\cite[Proposition 2.3 (6)]{Le15}}]\label{Prop_Le}
Let $\tau$ be the left-handed Dehn twist along $p_\al$. Then, for any $n\geq 1$
\begin{align*}
\beta\cdot T_n(p_\al)=q^{n}\tau^{n}(\beta)+q^{-n}\tau^{-n}(\beta).
\end{align*}
\end{prop}

Let $\overline{\al},\underline{\al}$ are $\cM$-tangle diagrams locally depicted in Figure \ref{Fig_local}. 
\begin{figure}[ht]
    \centering
    \includegraphics[width=240pt]{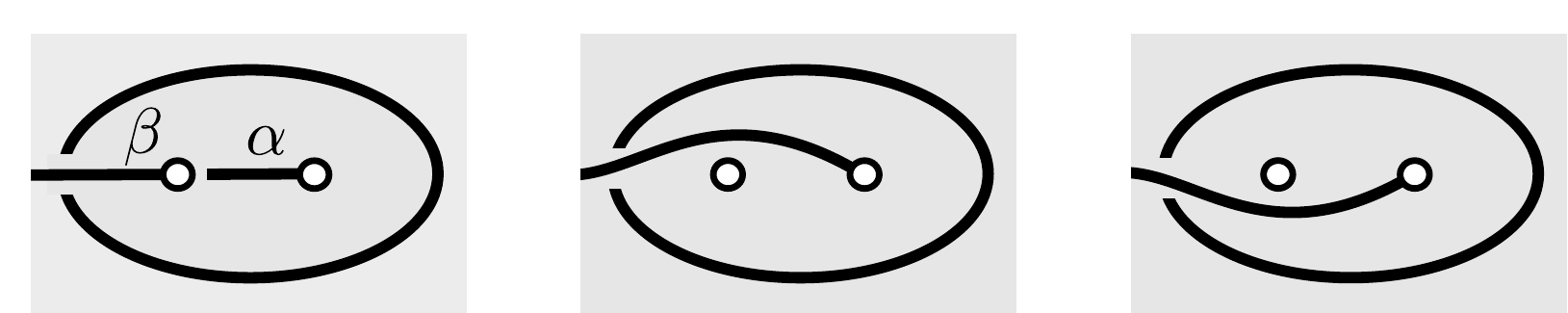}
    \caption{In each local picture, the punctures are $p'', p, p'$ from left to right. Left: $\beta\cdot p_\al\cdot \al$, Middle: $\overline{\al}\cdot p_\al$, Right: $\underline{\al}\cdot p_\al$}\label{Fig_local}
\end{figure}
\begin{prop}\label{Prop_twist}
Let $\tau$ be the left-handed Dehn twist along $p_\al$. Then, for any $n\geq 1$
\begin{align*}
\beta\cdot(S_n(p_\al)-S_{n-1}(p_\al))\al=q^{(2n+1)/2}\tau^{n}(\overline{\al})+q^{-(2n+1)/2}\tau^{-n}(\underline{\al}).
\end{align*}
\end{prop}
\begin{proof}
%In the proof, we use $\overline{\al},\underline{\al}$ instead of $\overline{\al},\underline{\al}$ respectively. 
It is easy to verify the case of $n=1$. 

Suppose that the case of $n=k\geq1$ holds. 
From $S_{k+1}(p_\al)-S_{k}(p_\al)=T_{k+1}(p_\al)-(S_{k}(p_\al)-S_{k-1}(p_\al))$, we have 
\begin{eqnarray*}
\beta\cdot(S_{k+1}(p_\al)-S_{k}(p_\al))\al&=&\beta\cdot(T_{k+1}(p_\al)-(S_{k}(p_\al)-S_{k-1}(p_\al)))\al\\
&=&(q^{1/2}\overline{\al}+q^{-1/2}\underline{\al})T_{k+1}(p_\al)-q^{(2k+1)/2}\tau^{k}(\overline{\al})-q^{-(2k+1)/2}\tau^{-k}(\underline{\al})\\
&=&q^{(2k+3)/2}\tau^{k+1}(\overline{\al})+q^{-(2n+3)/2}\tau^{-k-1}(\underline{\al}),
\end{eqnarray*}
where the last equality follows from Proposition \ref{Prop_Le} and $\overline{\al}=\tau(\underline{\al})$.
\end{proof}

Hence, with (\ref{rel:Cheby_even}), (\ref{rel;important}) and (\ref{formula:squared}), we have 
\begin{align}
\beta\cdot \barT_n(\al)=
\begin{cases}
q^{n/2}\tau^{n/2}(\beta)+q^{-n/2}\tau^{-n/2}(\beta)& \text{if $n$ is even}\\
q^{n/2}\tau^{(n-1)/2}(\overline{\al})+q^{-n/2}\tau^{-(n-1)/2}(\underline{\al})& \text{if $n$ is odd}
\end{cases}.\label{eq;once}
\end{align}

A $\bZ[q^{\pm1/2}]$-module basis of $\sSq(\Sigma)$ is \emph{positive} if the structure constants are in $\bZ_{\geq 0}[q^{\pm1/2}]$.
For a pair of normalized sequences $\mathbf{Q}$ 
is \emph{positive} for an oriented punctured surface $\Sigma$ over $\bZ[q^{\pm1/2}]$ if $\mathsf{B}_{\mathbf{Q}}$ is a positive basis of $\sSq(\Sigma)$. 

With similar definitions given in \cite{Le18} and \cite{LTY} for (Kauffman bracket) skein algebras, we have the following. 
\begin{thm}[{\cite[Theorem 2.4]{LTY}}, see also {\cite[Theorem 3.2]{Le18}}]\label{Thm_Le}
Suppose a normalized sequence $(Q_n(x))$ is positive for an oriented surface which is planar with at least 4 punctures or non-planar. 
Then, for its skein algebra, $(Q_n(x))\geq (\barT_n(x))$. 
Besides, $Q_1(x)=x$. 
\end{thm}

The following reformulates Theorem \ref{Intro_thm}. 
\begin{thm}\label{Thm_lower}
Suppose a pair $((Q_n(x)),(R_n(x)))$ of normalized sequences is positive for an oriented punctured surface $\Sigma$ which is planar with at least 4 punctures or non-planar with at least 3 punctures. 
Then, $(Q_n(x))\geq (\barT_n(x))$ and $(R_n(x))\geq (\barT_n(x))$.  
Besides, $Q_1(x)=R_1(x)=x$. 
\end{thm}
\begin{proof}Although Theorem \ref{Thm_Le} is for skein algebras, the claim and proof also work for Roger--Yang skein algebras. %since the difference is only Relation (C) and it does not affect to the inequality. 
Hence the claims $(Q_n(x))\geq (\barT_n(x))$ and $Q_1(x)=x$ follow from Theorem \ref{Thm_Le}.

To show the claims $(R_n(x))\geq (\barT_n(x))$ and $R_1(x)=x$, we follow an outline of the proof of Theorem 3.2 in \cite{Le18}. 
From the condition of $\Sigma$, there are 3 distinct punctures $p, p', p''$. 
Fix an arc $\al$ connecting $p$ and $p'$ and an arc $\beta$ connecting $p$ and $p''$ such that $\beta$ intersects $p_\al$ exactly once. See Figure \ref{Fig_local}.

By definition, $R_1(x)=x+a\ (a\in \bZ[q^{\pm1/2}])$. 

Note that $R_n(\al)$ is a $\bZ[q^{\pm1/2}]$-linear combination of $p_\al^n$ and $p_\al^n\al\ (n\in \bZ_{\geq0})$. 
Since $\sfB$ is a $\bZ[q^{\pm1/2}]$-basis of $\sSq(\Sigma)$ and $p_\al^n,p_\al^n\al\ (n\in \bZ_{\geq0})$ appear only in $P_n(\al)\ (n\in \bZ_{\geq0})$, there are $c_k\in \bZ[q^{\pm1/2}]$ such that 
$$ 
R_n(\al)=\sum_{k=0}^{n}c_k \barT_k(\al).
$$
Then, we have
\begin{eqnarray}\label{eq;lower}
R_1(\beta) R_n(\al)&=&aR_n(\al)+\beta \sum_{k=0}^{n}c_k \barT_k(\al)
=a R_n(\al)+c_0 \beta +\sum_{k=1}^{n}c_k \beta \barT_k(\al)\nonumber\\
&=&a R_n(\al)+c_0 \beta +\sum_{k=1}^{\lfloor (n+1)/2 \rfloor}c_{2k} \{q^{k}\tau^{k}(\beta)+q^{-k}\tau^{-k}(\beta)\}\nonumber\\
&\ &+\sum_{k=1}^{\lfloor (n+1)/2 \rfloor}c_{2k-1}\big( q^{(2k+1)/2}\tau^{k}(\overline{\al})+q^{-(2k+1)/2}\tau^{-k}(\underline{\al})\big),
\end{eqnarray}
where the last equality follows from (\ref{eq;once}). 
By inserting $x=R_1(x)-a$ to (\ref{eq;lower}), we have 
\begin{eqnarray*}
R_1(\beta) R_n(\al)&=& a R_n(\al)+c_0 (R_1(\beta)-a) +\sum_{k=1}^{\lfloor (n+1)/2 \rfloor}c_{2k} \{q^{k}\tau^{k}(R_1(\beta)-a)+q^{-k}\tau^{-k}(R_1(\beta)-a)\}\\
&\ &+\sum_{k=1}^{\lfloor (n+1)/2 \rfloor}c_{2k-1}\big(q^{(2k+1)/2}\tau^{k}(R_1(\overline{\al})-a)+q^{-(2k+1)/2}\tau^{-k}(R_1(\underline{\al})-a)\big)\\
&=& a R_n(\al)+c_0 R_1(\beta) +\sum_{k=1}^{\lfloor (n+1)/2 \rfloor}c_{2k} \{q^{k}\tau^{k}R_1(\beta)+q^{-k}\tau^{-k}R_1(\beta)\}\\
&\ &+\sum_{k=1}^{\lfloor (n+1)/2 \rfloor}c_{2k-1}\big(q^{(2k+1)/2}\tau^{k}R_1(\overline{\al})+q^{-(2k+1)/2}\tau^{-k}R_1(\underline{\al})\big)\\
&\ &-a\Big\{c_0+\sum_{k=1}^{\lfloor (n+1)/2 \rfloor}c_{2k} (q^{k}+q^{-k})+\sum_{k=1}^{\lfloor (n+1)/2 \rfloor}c_{2k-1}(q^{(2k+1)/2}+q^{-(2k+1)/2})\Big\}.
\end{eqnarray*}
From the positivity of $(R_n(x))$, the coefficients of $R_n(\al)$, $R_1(\beta)$, $\tau^{\pm k}(R_1(\beta))$, $\tau^{k}(R_1(\overline{\al}))$, $\tau^{-k}(R_1(\underline{\al}))$ $(k=1,\dots,\lfloor (n+1)/2 \rfloor)$ and the constant term are in $\bZ_{\geq 0}[q^{\pm1/2}]$. 
This implies that $a, c_k\ (k=0,1,\dots, n)$ are in $\bZ_{\geq 0}[q^{\pm1/2}]$ and 
$$
d:=-a\Big\{c_0+\sum_{k=1}^{\lfloor (n+1)/2 \rfloor}c_{2k} (q^{k}+q^{-k})+\sum_{k=1}^{\lfloor (n+1)/2 \rfloor}c_{2k-1}(q^{-(2k+1)/2}+q^{(2k+1)/2})\Big\}\in \bZ_{\geq 0}[q^{\pm1/2}]. 
$$
Since $-d\in \bZ_{\geq 0}[q^{\pm1/2}]$ from $a, c_k\in \bZ_{\geq 0}[q^{\pm1/2}]\ (k=0,1,\dots, n)$, we conclude that $d=0$. 
Since $\bZ[q^{\pm1/2}]$ is a domain and the second factor of $d$, i.e. the linear sum in the curly brackets is not 0, $d=0$ implies that $a=0$.  
\end{proof}
\begin{rmk}\label{Rmk_non-triv}
For the Muller skein algebra \cite{Mul}, 
we use $(x^{n})$ for arcs to give a lower bound of positive pairs of normalized sequences \cite{Le18}. 
However, for the Roger--Yang skein algebra, we use $\barT_n(x)$ to give a lower bound of positive pairs of normalized sequences.
\end{rmk}

\section{Partial evidence for positivity conjecture}
From (\ref{formula:squared}) and (\ref{rel:Cheby_SS}), for $m,n\geq 0$
\begin{eqnarray*}
\barT_{2n+1}(\al)\cdot \barT_{2m+1}(\al)&=&\big(S_{n}(p_\al)-S_{n-1}(p_\al)\big)\big(S_{m}(p_\al)-S_{m-1}(p_\al)\big)(p_\al+2)\\
&=&T_{n+m+1}(p_\al)+T_{|n-m|}(p_\al)=\barT_{2(n+m+1)}(\al)+\barT_{2|n-m|}(\al).
\end{eqnarray*}
Similar to this, from (\ref{rel:Cheby_TS}) and (\ref{rel:Chebyshev}), we have the positivity of $P_{n}(\al)\cdot P_{m}(\al)$ for other cases.

Let $\bD=\{(x,y)\in \bR^2\mid  \sqrt{x^2+y^2}<3\}$ and $\cM'=\{(1,0),(-1,0)\}$. We identify $\bD$ and a small neighborhood of an arc $\al$ connecting different punctures through some homeomorphism so that $\cM'\subset \cM$. Then, we treat $\cM$-tangle diagrams locally and refer to them as \emph{partial $\cM$-tangle diagrams} on $(\bD,\cM')$. 
\begin{figure}[ht]
    \centering
    \includegraphics[width=70pt]{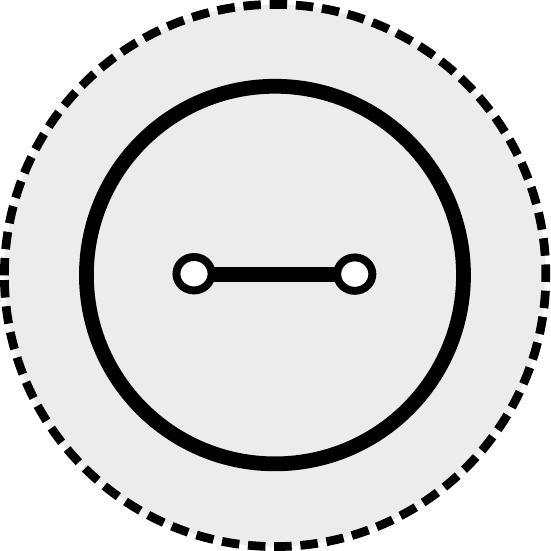}
    \caption{Picture of $(\bD,\cM')$. The arc is $\alpha$ and the simple closed curve is $p_\al$.}\label{Fig_disk}
\end{figure}

A $\bZ[q^{\pm 1/2}]$-linear combination of partial $\cM$-tangle diagrams on $(\bD,\cM')$ is \emph{symmetric} if it is invariant under each of replacing $q^{1/2}$ with $q^{-1/2}$ and taking the mirror image with respect to the $y$-axis. 

Put $z=\{(0,y)\mid  -3< y < 3\}$, $\al=\{(x,0)\mid  -1\leq x\leq 1\}$ and $p_\al=\{(x,y)\mid \sqrt{x^2+y^2}=2\}$. 
Let $\sigma$ be the left-handed half Dehn-twist along $p_\al$. 
Consider the partial $\cM$-tangle diagrams on $(\bD,\cM')$ 
\begin{align*}
\begin{array}{c}\includegraphics[scale=0.25]{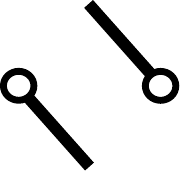}\end{array}=\{(x,y)\in \bD\mid  \text{$y=-3x+3\varepsilon$ with $0< |x|\leq 1$ and $\varepsilon=x/|x|$}\},\\
\begin{array}{c}\includegraphics[scale=0.25]{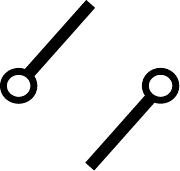}\end{array}=\{(x,y)\in \bD\mid  \text{$y=3x+3\varepsilon$ with $0< |x|\leq 1$ and $\varepsilon=-x/|x|$}\}.
\end{align*}
Then, the following is a partial evidence to establish the positivity conjecture of Roger--Yang skein algebras. 
\begin{thm}\label{Thm_formula}
For $n\geq 0$, 
$z\cdot \barT_{2n}(\al)$ is symmetric and 
\begin{align*}
z\cdot \barT_{2n+1}(\al)=
q^{2n+1}\sigma^{n}(\begin{array}{c}\includegraphics[scale=0.25]{draws/arcs_tr_bl.pdf}\end{array})+q^{-(2n+1)}\sigma^{-n}(\begin{array}{c}\includegraphics[scale=0.25]{draws/arcs_tl_br.pdf}\end{array})+\text{{\rm (symmetric part)}}. 
\end{align*}
In addition, the coefficients of the symmetric part are in $\bZ_{\geq 0}[q^{\pm1/2}]$.%{\color{red} when we restrict the basis ${\bf B}$ to $\bD$, the coefficients of the symmetric part with respect to ${\bf B}$}  
\end{thm}
\begin{proof}
\noindent {\bf Odd case.} 
By concrete computations, it is easy to see that 
\begin{align*}
&z\cdot \barT_1(\al)=q\begin{array}{c}\includegraphics[scale=0.25]{draws/arcs_tr_bl.pdf}\end{array}+q^{-1}\begin{array}{c}\includegraphics[scale=0.25]{draws/arcs_tl_br.pdf}\end{array}\\
&z\cdot \barT_3(\al)=q^3\sigma(\begin{array}{c}\includegraphics[scale=0.25]{draws/arcs_tr_bl.pdf}\end{array})+q^{-3}\sigma^{-1}(\begin{array}{c}\includegraphics[scale=0.25]{draws/arcs_tl_br.pdf}\end{array})+(q+q^{-1})(L+R)\al\\
&z\cdot \barT_5(\al)=q^5\sigma^{2}(\begin{array}{c}\includegraphics[scale=0.25]{draws/arcs_tr_bl.pdf}\end{array})+q^{-5}\sigma^{-2}(\begin{array}{c}\includegraphics[scale=0.25]{draws/arcs_tl_br.pdf}\end{array})+(L+R)\{(q^3+q^{-3})+(q+q^{-1})p_\al\}\al,
\end{align*}
where $L,R$ are partial $\cM$-tangle diagrams on $(\bD,\cM')$ given by $L=\{(2\cos \theta, 3\sin \theta)\mid \pi/2 < \theta < 3\pi/2 \}$ and $R=\{(2\cos \theta, 3\sin \theta)\mid -\pi/2 < \theta < \pi/2 \}$.

For $n\geq 3$, we show 
\begin{eqnarray}
&\hspace{0mm}&\hspace{-5mm}z\cdot \barT_{2n+1}(\al)\nonumber\\
&=&q^{2n+1}\sigma^{n}(\begin{array}{c}\includegraphics[scale=0.25]{draws/arcs_tr_bl.pdf}\end{array})+q^{-(2n+1)}\sigma^{-n}(\begin{array}{c}\includegraphics[scale=0.25]{draws/arcs_tl_br.pdf}\end{array})\nonumber\\
&\ &+(L+R)\{(q^{2n-1}+q^{-(2n-1)})\al+(q+q^{-1})\sum_{i=1}^{n-1}S_{i-1}(q^2+q^{-2})\barT_{2n-2i+1}(\al)\}.\label{eq;twice}
\end{eqnarray}
Indeed, it follows from the following computations; 
\begin{eqnarray*}
&\hspace{0mm}&\hspace{-5mm}z\cdot \barT_{2n+1}(\al)=z (p_\al\cdot \barT_{2n-1}(\al)-\barT_{2n-3}(\al))\\
&=&\{(q^2\sigma(z)+q^{-2}\sigma^{-1}(z)+(q+q^{-1})(L+R)\}\barT_{2n-1}(\al)-z\cdot \barT_{2n-3}(\al)\\
&=&q^{2n+1}\sigma^{n}(\begin{array}{c}\includegraphics[scale=0.25]{draws/arcs_tr_bl.pdf}\end{array})+q^{-(2n+1)}\sigma^{-n}(\begin{array}{c}\includegraphics[scale=0.25]{draws/arcs_tl_br.pdf}\end{array})+q^{2n-3}\sigma^{n-2}(\begin{array}{c}\includegraphics[scale=0.25]{draws/arcs_tr_bl.pdf}\end{array})+q^{-(2n-3)}\sigma^{-(n-2)}(\begin{array}{c}\includegraphics[scale=0.25]{draws/arcs_tl_br.pdf}\end{array})\\
&\ &+(q^2+q^{-2})(L+R)\{(q^{2n-3}+q^{-(2n-3)})\al+(q+q^{-1})\sum_{i=1}^{n-2}S_{i-1}(q^2+q^{-2})\barT_{2n-2i-1}(\al)\}\\
&\ &+(L+R)(q+q^{-1})\barT_{2n-1}(\al)-\big\{q^{2n-3}\sigma^{n-2}(\begin{array}{c}\includegraphics[scale=0.25]{draws/arcs_tr_bl.pdf}\end{array})+q^{-(2n-3)}\sigma^{-(n-2)}(\begin{array}{c}\includegraphics[scale=0.25]{draws/arcs_tl_br.pdf}\end{array})\\
&\ &\quad+(L+R)\big((q^{2n-5}+q^{-(2n-5)})\al+(q+q^{-1})\sum_{i=1}^{n-3}S_{i-1}(q^2+q^{-2})\barT_{2n-2i-3}(\al)\big)\big\}\\
&=&q^{2n+1}\sigma^{n}(\begin{array}{c}\includegraphics[scale=0.25]{draws/arcs_tr_bl.pdf}\end{array})+q^{-(2n+1)}\sigma^{-n}(\begin{array}{c}\includegraphics[scale=0.25]{draws/arcs_tl_br.pdf}\end{array})\\
&\ &+(L+R)\{(q^{2n-1}+q^{-(2n-1)})\al+(q+q^{-1})\sum_{i=1}^{n-1}S_{i-1}(q^2+q^{-2})\barT_{2n-2i+3}(\al)\}, 
\end{eqnarray*}
where the last equality follows from $S_k(x)=xS_{k-1}(x)-S_{k-2}(x)$ with $x=q^2+q^{-2}$. 
From $S_k(x+x^{-1})=\sum_{i=0}^k x^{k-2i}$, we conclude that the coefficients are in $\bZ_{\geq 0}[q^{\pm1/2}]$.

\noindent {\bf Even case.} 
The case of $n=0$ is trivial. By concrete computations, it is easy to see that 
\begin{align*}
&z\cdot \barT_2(\al)=(q^2\sigma+q^{-2}\sigma^{-1})(z)+(q+q^{-1})(L+R)\\
&z\cdot \barT_4(\al)=(q^4\sigma^2+q^{-4}\sigma^{-2})(z)+(q+q^{-1})(L+R)\{(q^2+q^{-2})+p_\al\}
\end{align*}

Similarly to the odd case, we obtain
\begin{align}
z\cdot \barT_{2n}(\al)=T_n(q^2\sigma+ q^{-2}\sigma^{-1})(z)+(q+q^{-1})(L+R)\sum_{i=0}^{n-1}\overline{T}_i(q^2+q^{-2})\overline{T}_{n-i-1}(p_\al) \label{eq;twice2}
\end{align}
for $n\geq3$, where $\barT_n(x)$ is defined by (\ref{def;barT}).
From $T_i(x+x^{-1})=x^{i}+x^{-i}$, the claim holds.
\end{proof}

\begin{rmk}
Similarly to the proof of Theorem \ref{Thm_formula}, we have \begin{align*}\barT_{2n+1}(\al)\cdot z =q^{-(2n+1)}\sigma^{n}(\begin{array}{c}\includegraphics[scale=0.25]{draws/arcs_tr_bl.pdf}\end{array})+q^{2n+1}\sigma^{-n}(\begin{array}{c}\includegraphics[scale=0.25]{draws/arcs_tl_br.pdf}\end{array})+\text{{\rm (symmetric part)}}, 
\end{align*}
where the symmetric part is exactly the same with that in Theorem \ref{Thm_formula}.  
\end{rmk}

\begin{cor}\label{Cor_transp}
If $\cR=\bC$ and $q^2$ is a primitive root of unity of order $N$, 
then $\barT_{N}(\al)$ is transparent. 
\end{cor}
\begin{proof}
In the case of $\cR=\bC$ with a root of unity $q\in \bC^\times$, (\ref{eq;twice}), (\ref{eq;twice2}) and (\ref{eq;once}) also hold. 
Moreover, a similar result of (\ref{eq;once}) restricted to $\bD$ also holds with $\beta'=\{(x,0)\in \bD\mid  -3<x\leq -1\}$ instead of $\beta$. 
This implies that, when $q^2$ is a primitive root of unity of order $N$, $\barT_{N}(\al)$ is transparent since $\barT_N(\al)$ is commutative with $\beta'$ and $z$ respectively. 
\end{proof}
\begin{rmk}
\begin{enumerate}
    \item The transparency of Chebyshev polynomials of the first kind is given in \cite[Corollary 2.3]{Le15}, \cite[Lemma 36]{BW16}. 
    %\item Corollary \ref{Cor_transp} implies that $\barT_{N}(\al)$ is in the center in the setting. This would be helpful to consider the whole center of the Roger--Yang skein algebra. 
    \item Helen Wong and her student have also shown the transparency with almost the same techniques. 
\end{enumerate}
\end{rmk}

\paragraph{{\bf (Non-)positivity for small punctured surfaces}}
For small surfaces, we have some observations on (non-)positivity of Roger--Yang skein algebras. 
Let $\Sigma_{0,2}$ (resp. $\Sigma_{0,3}$) denote the twice- (resp. thrice-)punctured sphere and put $\cR=\bZ[q^{\pm1/2}][p^{\pm1}\mid p\in \cM]$. 
In \cite{BPKW}, an explicit presentation of $\sS_q^{{\rm RY}}(\Sigma_{0,i})\ (i=2,3)$ is given. 
Theorem 2.3 in \cite{BPKW} implies that $\sS_q^{{\rm RY}}(\Sigma_{0,2})$ has the $\cR$-basis consisting of the empty set and the arc connecting the 2 punctures, and it is not a positive basis.  
Theorem 2.5 in \cite{BPKW} means that $\sS_q^{{\rm RY}}(\Sigma_{0,3})$ has the $\cR$-basis consisting of the empty set and the 3 arcs connecting distinct punctures, and it is a positive basis.

\end{document}